\documentclass[11pt,reqno]{amsart}

\usepackage[ansinew]{inputenc}

\usepackage{textcomp}

\usepackage{cite}

\usepackage{amsmath,amsfonts,amsthm,amssymb,amsxtra}
\usepackage[usenames,dvipsnames]{color}

\usepackage{bbm}
\usepackage{fancyhdr}
\usepackage[scanall]{psfrag}
\usepackage{graphicx}
\usepackage{ifthen}
\usepackage{pstricks}

\newtheorem{theorem}{Theorem}[section]
\newtheorem{proposition}[theorem]{Proposition}
\newtheorem{lemma}[theorem]{Lemma}
\newtheorem{corollary}[theorem]{Corollary}

\theoremstyle{definition}

\newtheorem{definition}[theorem]{Definition}

\theoremstyle{remark}

\newtheorem{remark}[theorem]{Remark}

\numberwithin{equation}{section}

\newcommand{\C}{\mathbb{C}}

\renewcommand{\epsilon}{\varepsilon}

\newcommand{\R}{\mathbb{R}}

\DeclareMathOperator{\dist}{dist}

\DeclareMathOperator{\im}{Im}

\DeclareMathOperator{\sgn}{sgn}
\DeclareMathOperator{\Tr}{Tr}
\DeclareMathOperator{\tr}{Tr}

\begin{document}

\title{The spectral shift function and Levinson's theorem for quantum star graphs}

\author{Semra Demirel}

\address{Semra Demirel, University of Stuttgart, Department of Mathematics, Pfaffenwaldring 57, D-70569 Stuttgart}
\email{Semra.Demirel@mathematik.uni-stuttgart.de}

\begin{abstract}
We consider the Schr\"odinger operator on a star shaped graph with $n$ edges joined at a single vertex. We derive an expression for the trace of the difference of the perturbed and unperturbed resolvent in terms of a Wronskian. This leads to representations for the perturbation determinant and the spectral shift function, and to an analog of Levinson's formula.
\end{abstract}

\maketitle

\section{Introduction and main results}

\subsection{Introduction}
This article focuses on the study of the spectral shift function and a Levinson theorem for Schr\"odinger operators on star shaped graphs. Quantum mechanics on graphs has a long history in physics and physical chemistry \cite{Pau,RS}, but recent progress in experimental solid state physics has renewed attention on them as idealized models for thin domains. A large literature on the subject has arisen and we refer, for instance, to the bibliography given in \cite{BK,EKKST}.

A \emph{star graph} is a metric graph $\Gamma$ with a single vertex in which a finite number $n\geq 2$ of edges $e_j$ are joined. We assume throughout that all edges $e_j$ are infinite and we identify $e_j=[0,\infty)$. We assume that the potential $V$ is a real-valued function on $\Gamma$ satisfying
\begin{equation}\label{assumhalb}
\int_{e_j} |V_j(x_j)| \,dx_j < \infty \quad \mbox{for all} \  1\leq j \leq n,
\end{equation}
where we denoted the restriction of $V$ to the edge $e_j$ by $V_j(x_j)=V(x)|_{e_j}$. Under this condition, we can define the Schr\"odinger operator
\begin{equation}\label{staroperator}
H \psi :=   - \psi'' + V \psi
\end{equation}
with continuity and Kirchhoff vertex conditions
\begin{equation}\label{kbc}
\psi_1(0) = \ldots = \psi_n(0) =: \psi(0), \quad \sum_{j=1}^n \psi'_j (0) = 0,
\end{equation}
as a self-adjoint operator in the Hilbert space $ L_2(\Gamma) = \oplus^n_{j=1} L_2(e_j)$. In \eqref{kbc} we denoted by $\psi_j$ the restriction of $\psi$ to the edge $e_j$. More precisely, we define the operator $H$ via the closed quadratic form
$$
h[\phi] :=	\int_\Gamma |\phi'(x)|^2 \,dx +	\int_\Gamma V(x)|\phi(x)|^2 \,dx\,,
$$
with form domain $d(h) = H^1(\Gamma)$ consisting of all continuous functions $\phi$ on $\Gamma$ such that $\phi_j \in H^1(e_j)$ for every $j$. If $V$ is sufficiently regular in a neighborhood of the vertex, then functions $\phi$ in the operator domain of $H$ satisfy the Kirchhoff vertex condition in \eqref{kbc}; otherwise this condition has to be interpreted in a generalized sense.

Our two main results are formulas for the spectral shift function and the perturbation determinant of $H$ with respect to the unperturbed operator $H_0$ (which is defined similarly as $H$, but with $V\equiv 0$) and an analog of Levinson's theorem. Special attention will be paid to the existence or absence of zero energy resonances.

There are several motivations for this study. The first one is the scattering theory of quantum graphs. While star graphs are certainly very special graphs, it is generally believed that they are a correct model example for a scattering process in the presence of a vertex. The direct and indirect scattering theory on star graphs has been studied in great detail in \cite{Ger} and \cite{Harm1}. Our results complement theirs and, in contrast to them, we advertise a more operator theoretic approach including, for instance, Fredholm determinants, trace class estimates and Kre{\v{\i}}n's resolvent formula.

A second motivation is a line of thought that goes back at least to Jost and Pais \cite{JoPa}; see also \cite{barry}, \cite{GeszMZ} and \cite{YOe}. In these works, a perturbation determinant, which is a Fredholm determinant in an infinite dimensional space, is shown to be equal to a much simpler determinant, typically in a finite dimensional space, such as a Wronski determinant. While such formulas appear in different set-ups, there seems to be no general method of knowing in advance the form of the `simpler determinant'. One of the achievements of this paper is to derive a new formula of this kind for a star graph.

A third motivation comes from the general interest in zero energy resonances because of their key role in several diverse problems of mathematical physics; for instance, the Efimov effect in many-body quantum mechanics \cite{Efi}, the time decay of wave functions \cite{JKato} and the convergence of `thick quantum graphs' \cite{Grie}, to name just a few. We also refer to the survey \cite{Boll, B}. In particular, we hope that our results will allow us to remove the non-resonance assumption in the recent dispersive estimates on star graphs \cite{MAN}; see also \cite{Weder} for similar bounds in the whole line case.

Finally, we note that the derivation of a Levinson theorem for a graph with a finite number of unbounded edges was mentioned as an open problem in \cite{childs} (who considered the discrete case). While the compact part of the graph still has to be better understood, our analysis explains how to deal with several unbounded edges and will be useful, we believe, in further developments in this direction.

\subsection{Main results}

To state our main result, namely a trace formula for the operator \eqref{staroperator} with vertex condition \eqref{kbc}, we need some notations.
By $H_{D,j}$ we denote the half-line Schr\"odinger operator with potential $V_j = V|_{e_j}$ and Dirichlet boundary condition at the origin.  The self-adjoint operator 
$$
H_{D,j} = -\frac{d^2}{dx_j^2} +V_j
$$
on $L_2(e_j)$ is associated with the quadratic form 
\begin{equation}\nonumber
h_{D,j}[\phi_j]:=  \int_{e_j} | \phi_j'(x_j) |^ 2 \,dx_j + \int_{e_j} V_j(x_j) | \phi_j(x_j) |^2  \,dx_j, \quad \phi_j \in H^{0,1}(e_j),
\end{equation}
\\
\noindent where the form domain is given by $d(h_{D,j})=H^{0,1}(e_j) = \{\phi_j \in H^{1}(e_j) :  \phi_j(0)=0 \}.$ 
If the condition \eqref{assumhalb} is satisfied, then the equation 
$$
-u'' + Vu = z u, \quad z =  \zeta^2
$$
has two particular solutions, the \textit{regular solution} $\varphi_j$ and the \textit{Jost solution} $\theta_j$. The first one is characterized by the conditions
$$
\varphi_j(0,\zeta) = 0 , \quad \varphi'_j(0,\zeta) = 1
$$
\noindent and the latter one by the asymptotics $ \theta_j(x,\zeta) = e^{ix\zeta} (1+o(1))$ as $|\zeta| \to \infty$. Both solutions are unique, see for instance \cite{Y}.
The \textit{Jost function} $w_j(\zeta) $ is defined as the Wronskian of the regular solution and the Jost solution and turns out to be $w_j(\zeta) = \theta_j(0,\zeta)$.

\vspace{1.5cm}
\begin{figure}[h!]
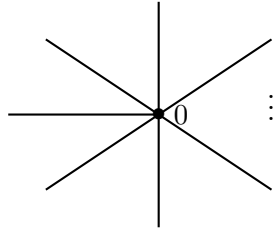

\qline(-2.0,0)(0,0)
\qline(0,1.5)(0,0)
\qline(1.5,-1.0)(0,0)
\qline(-1.5,1.0)(0,0)
\qline(-1.5,-1.0)(0,0)
\qline(1.5,1.0)(0,0)
\qline(0,-1.5)(0,0)
\rput(0.3,0){$0$}
\rput(1.5,0.2){\vdots}
\rput(0,0){$\bullet$}

\vspace*{2cm}
\caption{star graph $\Gamma$}
\end{figure}

Our first main result is 

\begin{theorem}\label{trfrstar}
Let $\Gamma$ be a star shaped graph and assume that \eqref{assumhalb} is satisfied for $1\leq j \leq n$. Then, for the Schr\"odinger operator \eqref{staroperator} on $L_2(\Gamma)$ with Kirchhoff vertex condition \eqref{kbc}, the following trace formula holds,
\begin{equation}\label{resolventendiffstern}
\tr \left(  (H_0 - \zeta^2)^{-1} - (H- \zeta^2)^{-1} \right) =  \frac{1}{2 \zeta} \frac {d}{d \zeta} \ln \left( \frac{K(\zeta)}{\zeta} \prod_{j=1}^n w_j(\zeta)  \right), \quad \im \zeta >0,
\end{equation}

\noindent where $K(\zeta) = \sum_{j=1}^n \theta'_j(0, \zeta) / \theta_j(0,\zeta)$ and $w_j(\zeta) = \theta_j(0,\zeta)$.
\end{theorem}

\begin{remark}\label{Jost-Pais}
We note that identity \eqref{resolventendiffstern} is equivalent to the identity 
$$
\tr \left(  (H_0 - \zeta^2)^{-1} - (H- \zeta^2)^{-1} \right) =   \frac{1}{2 \zeta} \left( \sum_{j=1}^n \frac{ \frac{d}{d\zeta} w_j(\zeta)}{w_j(\zeta)} + \frac{\frac{d}{d\zeta} K(\zeta)}{K(\zeta)} - \frac{1}{\zeta} \right),
$$
which should be compared with the classical result \cite{ JoPa , BF}, see also \cite{Y, barry},
\begin{equation}\label{classicresult}
\tr \left( (H_{D,j,0} - \zeta^2)^{-1} - (H_{D,j}- \zeta^2)^{-1} \right) = \frac{\frac{d}{d\zeta} w_j(\zeta)}{2 \zeta w_j(\zeta)}.
\end{equation}

\end{remark}

From equation \eqref{resolventendiffstern}, we conclude in Section \ref{PDSSF} an explicit expression for the perturbation determinant $D(z)$ and the spectral shift function $\xi( \lambda; H, H_0)$. We recall that the spectral shift function can be characterized by the formula 
$$
\tr \left( f(H) -f(H_0) \right) = \int_{- \infty}^\infty \xi (\lambda; H, H_0) f'(\lambda) \,d\lambda ,
$$
for any $f \in C_0^\infty (\R)$ (together with the condition $\xi(\lambda; H,H_0)=0$ for $\lambda < \inf \sigma(H)$). An extension of this formula for a broader class of functions, as well as several equivalent definitions are discussed in Section $3$.

In Section \ref{levii} we study the low-energy asymptotics of $D(z)$ as $|z| \to 0$. This allows us to prove an analog of Levinson's formula for the star graph. 
We say that the operator $H$ on $L_2(\Gamma)$, given in \eqref{staroperator}, has a \emph{resonance} at $\zeta = 0$ if the equation $-u'' + Vu = 0$
has a non-trivial bounded solution satisfying the continuity and Kirchhoff conditions. By definition, the \emph{multiplicity of the resonance} is the dimension of the corresponding solution space.

\begin{theorem}\label{levinsonstar}
Assume that  
\begin{equation}\label{firstmoment}
\int_{e_j} (1+x) |V_j(x)| \,dx < \infty \quad \mbox{for all} \  1 \leq j \leq n,
\end{equation}
is satisfied and, if $\zeta =0$ is a resonance of multiplicity one, assume that 
\begin{equation}\label{secondmoment}
\int_{e_j} (1+x^2) |V_j(x)| \,dx < \infty \quad \mbox{for all} \ 1 \leq j \leq n.
\end{equation}
Then,
\begin{equation}
\lim_{\lambda \to 0+} \xi (\lambda) = - \left(N+ \frac{m-1}{2}\right),
\end{equation}
where $N$ is the number of negative eigenvalues of $H$ and where $m \geq 1$ is the multiplicity if $\zeta =0$ is a resonance and $m=0$ if $\zeta =0$ is not a resonance.
\end{theorem} 

\begin{remark}
We know from Bargmann's bound that $N < \infty$ if \eqref{firstmoment} is satisfied, \cite{Barg}. We also know that $\lim_{\lambda \to 0-} \xi(\lambda) = -N$, which is an easy consequence of the definition of the spectral shift function.
\end{remark}

\section{A Trace formula for Star Graphs}\label{ATf}

In this section, our goal is to prove a trace formula for star graphs. More precisely, we will find an expression for $\tr (R(z) - R_0(z) )$ in terms of the Jost solutions $\theta_j$ on the edges $e_j$. Here and in the following we write $R(z) = (H-z)^{-1}$ and $R_0(z) = (H_0 - z)^{-1}$ for the perturbed and unperturbed resolvent, respectively.
When deriving an expression for the resolvent $R(z)$, we will make use of Kre{\v{\i}}n's formula for which we refer to \cite{AGHH} and, in particular, to an article by Exner \cite{Ex} where this formula was used in a similar context. Thereby, we need to decouple the operator $H$ which we achieve by imposing Dirichlet vertex conditions on each edge $e_j$, i.e.,
$\psi_j(0) = 0$ for all $ 1 \leq j \leq n.$ Then the operator \eqref{staroperator} is decoupled and the half-lines are disconnected. We denote the decoupled operator by 
$$ 
H_{\infty} = \bigoplus_{j=1}^n H_{D,j}
$$
and its resolvent by $R_{\infty}(z) = (H_{\infty} - z)^{-1}$.
In what follows, we will skip for simplicity the indices at the coordinates and use the notation $\psi_j( x):= \psi_j(x_j) ,\ 1 \leq j \leq n,$ for a function defined on the edge $e_j$ of $\Gamma.$ 
\vspace{0.3cm}

\begin{proof}[Proof of Theorem \ref{trfrstar}]
It is a well-known fact, see e.g. \cite{Y}, that under assumption \eqref{assumhalb} for all $z = \zeta^2$ such that $\im z \neq 0$ and $w_j(\zeta) \neq 0,$ the resolvent $R_{D,j}(z) = (H_{D,j} -  z)^{-1}$ is an integral operator with kernel
$$
R_{D,j}(x,y; z) := \dfrac{\varphi_j(x, \zeta) \theta_j (y,\zeta)}{w_j(\zeta)}, \quad x \leq y,\ \ \zeta = z^{1/2},
$$
and $R_{D,j}(x,y; z) = R_{D,j}(y,x; z).$
Hence, the resolvent $R_\infty (z)$ is a matrix  integral operator with the kernel 
$$
R^{\infty}_{j, \ell} (x,y; z) := \delta_{j, \ell} R_{D,j}(x,y,z), \quad 1 \leq j, \ell \leq n.
$$

Having the resolvent kernel $R^{\infty}_{j, \ell} $ of the decoupled operator $H_{\infty}$, we can use Kre{\v{\i}}n's formula \cite{AGHH} to determine the kernel of the resolvent $R (z).$ Let $\rho(H)$ be the resolvent set of the operator $H$ and $\rho(H_0)$ the resolvent set of $H_0$.
The formula states that for any $\zeta$, such that $\im \zeta \geq 0$ and $z= \zeta^2 \in \rho(H_{\infty}) \cap \rho(H)$, the resolvent $R(z)$ is a matrix integral operator with kernel $ R_{j,\ell}(x,y; z) = R^{\infty}_{j, \ell} (x,y; z) + \lambda_{j \ell} \theta_j(x,\zeta) \theta_\ell (y,\zeta).$ In order to determine the coefficients $\lambda_{j \ell}$ we proceed as follows.
For any $f = (f_1(x), \ldots ,f_n(x) )^T \in L_2(\Gamma),$ the function $\psi(x) := \int  R(x,y; z) f(y) \,dy$ has to satisfy the equation $H \psi = \zeta^2 \psi +f$ and the Kirchhoff vertex condition. This leads to a system of $n$ linear equations for the coefficients $\lambda_{j \ell}.$ It turns out that $\lambda_{j \ell}= ( - K(\zeta) \theta_j(0, \zeta) \theta_{\ell}(0, \zeta))^{-1},$ with $K(\zeta) = \sum_{j=1}^n \theta'_j(0, \zeta) / \theta_j(0,\zeta)$, see also \cite{Ex}.
Thus,

\begin{equation}\label{krein}
R_{j,\ell}(x,y; z) := R^{\infty}_{j, \ell} (x,y; z) - \frac{\theta_j(x,\zeta) \theta_\ell (y,\zeta)}{K(\zeta) \theta_j(0, \zeta) \theta_{\ell}(0, \zeta)}.
\end{equation}

This representation allows us to compute $\tr (R(z) - R_0(z)).$ First, we note that the operator $R(z) - R_0(z)$ is a trace class operator. This can be seen as follows. As the quotient in \eqref{krein} is a perturbation of finite rank, we only have to show that the difference $R_\infty(z) - R_\infty^{(0)}(z)$ is trace class. Here $ R_\infty^{(0)}(z)$ is the resolvent of the unperturbed decoupled operator $H_{\infty}^{(0)} = \bigoplus_{j=1}^n \left( -d^2/dx^2 \right)$ on $\bigoplus_{j=1}^n L_2(e_j)$. Similarly, we denote by $R_{D,j}^{(0)}(z)$ the resolvent of the unperturbed operator $H_{D,j}^{(0)} = -d^2/dx^2$ on $L_2(e_j)$. Under condition \eqref{assumhalb} the operator $\sqrt{|V_j|} \left( R_{D,j}^{(0)}(z) \right)^{\alpha}$ is Hilbert-Schmidt for all $\alpha > 1/4$ and all $1 \leq j \leq n$, as can be easily checked (see e.g. \cite{Y}, Lemma 4.5.1). Hence, the Birman-Schwinger operator $ \sqrt{|V_j|} R_{D,j}^{(0)}(z) \sqrt{V_j}$, with $ \sqrt{V_j} := \sgn (V_j)  \sqrt{|V_j|}$, is trace class and has for $z \in \rho(H_{D,j})$ no eigenvalue $-1$. Thus, the following resolvent identity for the half-line Schr\"odinger operator holds,
$$
R_{D,j}(z) -R_{D,j}^{(0)}(z) = - R_{D,j}^{(0)} (z) \sqrt{V_j} \left( \mathbbm{1} + \sqrt{|V_j|} R_{D,j}^{(0)}(z) \sqrt{V_j} \right)^{-1} \sqrt{|V_j|} R_{D,j}^{(0)}(z).
$$
It follows from this resolvent identity that $R_\infty(z) - R_\infty^{(0)}(z)$ is a trace class operator. In view of \eqref{krein} it follows that also $R(z) - R_0(z)$ is a trace class operator and
\begin{eqnarray} \label{r0}
\tr (R(z) - R_0(z)) 
&=& \sum_{j=1}^n \int_{e_j} \left( R_{D,j}(x,x,z) - R^{(0)}_{D,j}(x,x,z) \right) \,dx \\ \nonumber
& & \quad + \sum_{j=1}^n \int_{e_j} \left( - \frac{\theta_j^2(x,\zeta)}{\theta_j^2(0,\zeta) K(\zeta)} +  \frac{ e^{2ix \zeta}}{ n i \zeta} \right) \,dx.
\end{eqnarray}
The computation of the first integral on the right-hand side is the classical Jost-Pais result \cite{ JoPa} recalled in Remark \ref{Jost-Pais},
\begin{equation}\label{r1}
 \int_{e_j} \left( R_{D,j}(x,x,z) - R_{D,j}^{(0)}(x,x,z) \right) \,dx = - \frac{\dot{w}_j(\zeta)}{ 2 \zeta w_j(\zeta)}.
\end{equation}
Here the derivative with respect to $\zeta$ is denoted by a dot, "$\cdot = d/d\zeta$".
To compute the second integral, we use the following equation which is true for any two arbitrary solutions of the equation $H_{D,j} \psi_j= \zeta^2 \psi_j$, namely
\begin{equation}\nonumber
2 \zeta u_j(x,\zeta) v_j (x,\zeta) = (u'_j(x,\zeta) \dot{v}_j(x,\zeta) - u_j(x,\zeta) \dot{v}'_j(x,\zeta))'.
\end{equation}
Applying this identity to $u_j=v_j=\theta_j$, we get
$$
\int_{\R_+} \frac{\theta_j^2(x,\zeta)}{K(\zeta) \theta_j^2(0,\zeta) } \,dx  = \frac{\left[ \theta_j'(x,\zeta) \dot{\theta}_j(x,\zeta) - \theta_j(x,\zeta) \dot{\theta}'_j (x,\zeta) \right]_0^{\infty}}{2 \zeta K(\zeta) \theta_j^2(0,\zeta)}.
$$
First, we consider the case of compactly supported potential $V_j$. Then, for large $x$ the Jost solution for the half-line Schr\"odinger operator $H_{D,j}$ is given by $\theta_j(x,\zeta) = e^{i \zeta x}$ and we have 
$$
\theta'_j (x,\zeta) = i \zeta e^{i \zeta x}, \quad \dot{\theta}_j (x,\zeta) = ix e^{i \zeta x}, \quad \dot{\theta}'_j (x,\zeta) = (i-x \zeta) e^{i \zeta x}.
$$
Therefore, for large $x$,
$$
\theta'_j (x,\zeta) \dot{\theta}_j (x,\zeta) - \theta_j (x,\zeta) \dot{\theta}'_j (x,\zeta) = i e^{2i\zeta x}.
$$
Note that the function $e^{2 i \zeta x}$ vanishes for $x \to \infty$, as $ \im \zeta >0$. We therefore get 

\begin{equation}\label{r2}
 \sum_{j=1}^n \int_{\R_+} - \frac{\theta_j^2(y,\zeta)}{\theta_j^2(0,\zeta) K(\zeta)} \,dy =  \sum_{j=1}^n \frac{ \theta'_j (0,\zeta) \dot{\theta}_j(0,\zeta) - \theta_j(0,\zeta) \dot{\theta}'_j (0,\zeta)  }{2 \zeta K(\zeta) \theta_j^2(0,\zeta)} = - \frac{ \frac{d}{d\zeta} K(\zeta)}{2 \zeta K(\zeta)}.
\end{equation}
By density arguments (see [Prop. 4.5.3,\cite{Y}]), based on the fact that $\sqrt{|V|} (H_{D,0}+\mathbbm{1})^{-1/2}$ is a Hilbert-Schmidt operator under condition \eqref{assumhalb}, the result can be extended to all potentials $V_j$ satisfying \eqref{assumhalb}.
Similarly, we consider the case $V= 0$ and obtain
\begin{equation}\label{r3}
\sum_{j=1}^n \int_{\R_+}    \frac{ e^{2iy\zeta}}{ n i \zeta} \,dy  = \frac{1}{2\zeta^2}.
\end{equation}
Combining \eqref{r1}, \eqref{r2} and \eqref{r3} with \eqref{r0}, we finally arrive at

\begin{eqnarray}
\tr (R(z) -R_0(z) ) &=& \frac{1}{2\zeta} \left( - \frac{\dot{K} (\zeta)}{K(\zeta)} + \frac 1\zeta - \sum_{j=1}^n \frac{\dot{w}_j(\zeta)}{w_j(\zeta)} \right) \\ \nonumber
&=&  \frac{1}{2\zeta} \left(  - \frac{d}{d\zeta} (\ln K(\zeta)) + \frac{d}{d\zeta} ( \ln \zeta) - \sum_{j=1}^n \frac{d}{d\zeta} (\ln w_j(\zeta) ) \right) \\ \nonumber
&=& - \frac{1}{2\zeta} \left( \frac{d}{d\zeta} \ln \left( \zeta^{-1} K(\zeta) \prod_{j=1}^n w_j(\zeta) \right) \right).
\end{eqnarray}
This is the claimed formula.
\end{proof}

\section{The perturbation determinant and the spectral shift function} \label{PDSSF}

Identity \eqref{resolventendiffstern} implies an explicit expression for the perturbation determinant
$$
D(z) :=  \mbox{det} (\mathbbm{1} +\sqrt{V} R_0(z) \sqrt{|V|}),\quad  z\in \rho(H_0),
$$
where $\sqrt{V} = (\sgn V) \sqrt{|V|}$. Strictly speaking this is the \textit{modified} perturbation determinant, nevertheless we shall refer to it simply as the perturbation determinant in what follows.  
Note that under the assumption \eqref{assumhalb} the perturbation determinant $D(z)$ is well-defined since the operator $\sqrt{|V|} (H_0 - z)^{-1/2}$ is Hilbert-Schmidt and therefore $\sqrt{V} R_0(z) \sqrt{|V|}$ is trace class. This follows as above from the fact that $\sqrt{|V_j|} \left( R_{D,j}^{(0)}(z)\right)^{ 1/2}$ is Hilbert-Schmidt for all $1\leq j \leq n$ together with \eqref{krein} for $V \equiv 0$ as the corresponding second term on the right-hand side of \eqref{krein} is of finite rank.

Furthermore, a simple computation shows, see also (0.9.36) \cite{Y2}, that the perturbation determinant is related to the trace of the resolvent difference by
\begin{equation}\label{wellstar}
D^{-1} (z) D'(z) = \Tr(R_0(z) - R(z)), \quad z\in \rho(H_0) \cap \rho(H).
\end{equation}
Hence, in view of Theorem \ref{trfrstar} we conclude that 
$$
D^{-1} (z) D'(z) =  \dfrac{ \frac{d}{dz} \left( z^{-1/2} K(z^{1/2}) \prod_j w_j(z^{1/2}) \right) }{z^{-1/2} K(z^{1/2}) \prod_j w_j(z^{1/2})},
$$
here we choose the square root of $z$ such that $\im z^{1/2} > 0$.
From which it follows that $D(z) =  C z^{-1/2} K(z^{1/2}) \prod_{j=1}^n w_j(z^{1/2}),$ for some $C \in \C.$
The coefficient $C$ is fixed by the asymptotics of a perturbation determinant, namely
\begin{equation}\label{limitDstar}
\lim_{ |\im z | \to \infty} D(z) = 1.
\end{equation}
This asymptotics is true if the operator $|V|^{1/2} (H_0 - z )^{-1/2}$ is Hilbert-Schmidt, see e.g. (0.9.37,\cite{Y2}). 
As $|\zeta| \to \infty, $ we have \cite{Y,SU}
\begin{equation}\label{wK}
w_j(\zeta) = \theta_j(0,\zeta) = 1+ O(|\zeta|^{-1}) \quad \mbox{and} \quad K(\zeta) = n i \zeta +O(1).
\end{equation}
This implies that $C= 1/i n.$
Thus, we have proved
\begin{corollary}
Assume that \eqref{assumhalb} is satisfied. Then, for $z\in \rho(H) $, the perturbation determinant of $H$ with respect to $H_0$ is given by
\begin{equation}\label{Dstar}
D(z) = \frac{  K(z^{1/2})}{ i n z^{1/2}} \prod_{j=1}^n w_j(z^{1/2}),
\end{equation}
where $\im z^{1/2} > 0$. 
\end{corollary}

Our next goal is to determine an explicit expression for the spectral shift function $\xi(\lambda; H, H_0)$ for the pair of operators $H, H_0$ in $L_2(\Gamma)$. If \eqref{assumhalb} is satisfied, then by the argumentation above the resolvent difference $R(z) -R_0(z) $ is trace class for all $z \in \rho(H)$. In this case it is known from general theory that for all $-c < \inf \sigma(H)$ there exists a real-valued function $\xi_c(\lambda)$ for the pair of operators $R(-c),R_0(-c)$ such that the relation 
\begin{equation}\label{trssf}
\tr \left( f(R(-c)) -f(R_0(-c)) \right) = \int_{-\infty}^\infty \xi_c (\lambda) f'(\lambda) \,d\lambda
\end{equation}
is true for all functions $f \in C_0^\infty(\R)$. This formula goes back to Lifshits \cite{Lif}. 

The spectral shift function for the pair $H, H_0$ is defined by the relation
\begin{equation}\label{definessf}
\xi(\lambda; H, H_0) := - \xi_c ((\lambda+c)^{-1}, R(-c), R_0(-c))
\end{equation}
for $\lambda > -c$ and $\xi(\lambda; H,H_0):=0$ for $\lambda \leq -c$. It can be shown that this definition is independent of the choice of $c$. By a change of variables the formula \eqref{trssf} for the pair $R,R_0$ can then be transformed into a formula for the pair $H,H_0$ and yields
\begin{equation}\label{trSSFstar}
\tr \left( f(H) -f(H_0) \right) = \int_{-\infty}^\infty \xi (\lambda; H, H_0) f'(\lambda) \,d\lambda ,
\end{equation}
for all functions $f \in C_0^\infty(\R)$.

The next theorem among other things extends the class of admissible functions $f$ in this trace formula.

\begin{theorem}\label{meins}
Let $H$ be the Schr\"odinger operator in $L_2(\Gamma)$ given in \eqref{staroperator} with the Kirchhoff vertex condition and $H_0 = -d^2/dx^2$ the corresponding unperturbed operator. Assume that condition \eqref{assumhalb} is satisfied. Then, the spectral shift function for the pair of operators $H,H_0$ is given by
$$
\xi (\lambda; H,H_0) = \pi^{-1} \lim_{\varepsilon \to 0+} \arg D(\lambda +i \varepsilon),
$$ 
where $\arg D(z) = \im \ln D(z)$ is defined via $\ln D(z) \to 0$ as $ \dist (z, \sigma(H_0)) \to \infty.$

Moreover,
\begin{equation}\label{lnSSFstar}
\ln D(z) = \int_{-\infty}^\infty \xi(\lambda; H,H_0) (\lambda-z)^{-1} \,d\lambda, \quad z \in \rho(H_0) \cap \rho(H),
\end{equation}
and \eqref{trSSFstar} holds provided $f$ has two locally bounded derivatives and for any $\varepsilon > 0$, $m > -1/2$  as $\lambda \to \infty$,
\begin{equation}\label{boundedsol}
f'(\lambda) = O( \lambda^{-m-1-\varepsilon}) , \quad f''(\lambda^{-2m-2}).
\end{equation}
Finally for $m > -1/2$,
$$
\int_{-\infty}^\infty | \xi(\lambda; H,H_0)| (1+|\lambda|)^{-m-1} \,d\lambda < \infty.
$$

\end{theorem}
In the proof of Theorem \ref{meins} we make use of some results from abstract scattering theory which we collect for the reader's convenience in the following proposition. These results can be found e.g. in [Chapter 0.9, \cite{Y}].

\begin{proposition}\label{SSF}
Let $h$ and $h_0$ be lower semi-bounded operators on a Hilbert space and $v= h-h_0$. Assume that the operator $\sqrt{ | v | } \left( r_0(-c) \right)^{1/2}$ is Hilbert-Schmidt for some $c< \inf (\sigma (h) \cup \sigma (h_0) ) $, where $r_0(-c)=(h_0+c)^{-1}$ and $r(-c)=(h+c)^{-1}$.
Moreover, assume that 
\begin{equation}\label{toshow}
\int_c^\infty \| r(-t) -r_0(-t) \|_1 t^{-m} \,dt < \infty
\end{equation}
for some $m \in (-1,0)$. Let the spectral shift function $\xi(\lambda;h,h_0)$ be defined as above and the (modified) perturbation determinant by
$$
d(z) = \det ( \mathbbm{1} + \sgn v \sqrt{|v|} r_0(z) \sqrt{|v|} ), \quad z \in \rho(h_0).
$$
Then, $d(z) \to 1 \quad \mbox{as} \ \dist (z, \sigma(h_0)) \to \infty$ and 
\begin{equation}\label{argssf}
\xi (\lambda; h,h_0) = \pi^{-1} \lim_{\varepsilon \to 0+} \arg d(\lambda +i \varepsilon),
\end{equation}
where $\arg d(z) = \im \ln d(z)$ is defined via $\ln d(z) \to 0$ as $ \dist (z, \sigma(h_0)) \to \infty.$
Moreover,
\begin{equation}\label{lnSSF}
\ln d(z) = \int_{-\infty}^\infty \xi(\lambda; h,h_0) (\lambda-z)^{-1} \,d\lambda, \quad z \in \rho(h_0) \cap \rho(h),
\end{equation}
and
\begin{equation}\label{trSSF}
\tr ( f(h) -f(h_0)) = \int_{-\infty}^\infty  \xi(\lambda; h,h_0) f'(\lambda) \,d\lambda
\end{equation}
provided $f$ is as in \eqref{boundedsol}.
Finally,
$$
\int_{-\infty}^\infty | \xi(\lambda; h,h_0)| (1+|\lambda|)^{-m-1} \,d\lambda < \infty.
$$
\end{proposition}
In what follows, we prove that for the case of star graphs $\Gamma$ the condition \eqref{toshow} is satisfied for all $m> -1/2$. Then, by Proposition \ref{SSF} the assertions of Theorem \ref{meins} will follow.
\noindent We need to compute the trace norm of a rank two operator.
\begin{lemma}
Let  $\mathcal{H}$ be a Hilbert space and $f,g \in \mathcal{H}$. Further, assume that $\mathcal{R} = (\cdot , f)f - (\cdot , g)g$ is an operator of rank two on $\mathcal{H}.$ Then, the trace norm of $\mathcal{R}$ is given by
\begin{equation}\label{rank2norm}
\| \mathcal{R} \|_1 = \left( ( \|f\|^2 +\|g\|^2 ) ^2 - 4 |(f,g)|^2 \right)^{1/2}.
\end{equation}
\end{lemma}

\begin{proof}
We may assume that $f \neq 0$, for otherwise formula \eqref{rank2norm} is obvious.
We construct an orthonormal bases on Ran$({\mathcal{R}})$ by applying the Gram Schmidt process,
$$
\mathcal{B} = \{ f/ \|f\| , \left( g-(f,g) f/ \|f\|^2 \right) / \left( \|g\|^2 - |(f,g)|^2/ \|f\|^2 \right)^{-1/2} \}.
$$
The operator $\mathcal{R}$ is described by a $2 \times 2$ matrix $M=(m)_{k \ell}$, where
$$
m_{11} = \|f\|^2 - |(f,g)|^2 / \|f\|^2, \quad m_{22} = - \| g\|^2 + |(f,g)|^2 / \|f \|^2,
$$
$$
m_{12}= \overline{m_{21} }= -\left((f,g) / \|f\|\right) \left( \|g\|^2 - |(f,g)|^2/ \|f\|^2 \right)^{1/2}.
$$
The singular values of $M$ turn out to be 
$$
s_1 = \frac 12 ( \|f \|^2 - \|g \|^2 + \omega ) ,\quad s_2 = \frac 12 (\omega - \|f\|^2 + \|g \|^2 ),
$$
where $\omega := \left( ( \|f\|^2 +\|g\|^2 ) ^2 - 4 |(f,g)|^2 \right)^{1/2}.$ Hence, $\| \mathcal{R} \|_1 = s_1  +s_2 = \omega$. This proves \eqref{rank2norm}.
\end{proof}

\begin{lemma}\label{conditioncheck}
Let $H$ be the Schr\"odinger operator in $L_2(\Gamma)$ given in \eqref{staroperator} with the Kirchhoff vertex condition and $H_0 = -d^2/dx^2$ the corresponding unperturbed operator. Assume that condition \eqref{assumhalb} is satisfied. Then, the resolvents $R(z)$ and $R_0(z)$ satisfy for all $\varepsilon >0$ and for $t$ big enough,
$$
\| R(-t) - R_0(-t) \|_1 \leq C_\varepsilon t^{-3/2+ \varepsilon}.
$$
\end{lemma}
\begin{proof}
By adding zero we estimate the trace norm by
\begin{equation}\label{trnormrank2}
\| R(-t) -R_0(-t) \|_1 \leq \| \mathcal{R} \|_1 + \| R_\infty(-t) - R_\infty^{(0)} (-t) \|_1,
\end{equation}
where $\mathcal{R}:= R(-t) - R_\infty(-t) + R_\infty^{(0)} (-t) - R_0(-t)$ is an operator of rank two. 
The second norm in the right-hand side of \eqref{trnormrank2} can be estimated by $  \| R_\infty(-t) - R_\infty^{(0)} (-t) \|_1 \leq c t^{-3/2 + \varepsilon}$, see \cite{Y}, Lemma 4.5.6.
Thus it remains to bound the norm of the rank two operator $\mathcal{R}$.
From Kre{\v{\i}}n's formula \eqref{krein} we have
$$
( R_0(x_k,x_{\ell},-t) - R_\infty^{(0)} (x_k,x_{\ell},-t) )_{k \ell} = f_k(x_k) f_{\ell}(x_{\ell}), \quad f_k(x_k) = \frac{e^{-\sqrt{t}x_k}}{ \sqrt{n} t^{1/4}} ,
$$
and hence, 
$$
R_0(-t) -  R_\infty^{(0)} (-t)  = ( \cdot, f)f, \quad f=(f_1, \ldots , f_n)^T.
$$
Further,
$$
( R(x_k,x_{\ell},-t)  - R_\infty(x_k,x_{\ell},-t))_{k \ell} = g_k(x_k) g_{\ell}(x_{\ell}), \quad g_k(x_k) = \frac{\theta_k(x_k, i\sqrt{t})}{\sqrt{ - K(i \sqrt{t})}\theta_k(0,i \sqrt{t})} ,
$$
and
$$
R(-t) - R_\infty(-t)  =  (\cdot, g) g,\quad g = (g_1, \ldots , g_n)^T.
$$
In view of \eqref{rank2norm} we have 
\begin{equation}
 \| R(-t) - R_\infty(-t) + R_\infty^{(0)} (-t) - R_0(-t) \|_1 = \left( ( \|f\|^2 +\|g\|^2 ) ^2 - 4 |(f,g)|^2 \right)^{1/2}.
\end{equation}
In the remaining part we show that $ (\|f\|^2 + \|g\|^2)^2 - 4 |(f,g)|^2 = O( t^{-3})$ as $t \to \infty$. Let us set $g= f+h$ and note that $h$ is a real-valued function. Then by applying the Cauchy-Schwarz inequality and the arithmetic inequality,
\begin{equation}\label{rank2}
 (\|f\|^2 + \|g\|^2)^2 - 4 |(f,g)|^2 = 4 ( \|f\|^2 +(f,h) ) \|h\|^2 + \|h\|^4 \leq 6 \|h\|^2 \|f\|^2 + 3 \|h\|^4 .
\end{equation}
We compute
\begin{equation}\label{f^2}
\|f\|^2 = \sum_{j=1}^n \int_0^\infty \frac{e^{-2\sqrt{t} x_j}}{n\sqrt{t}}  \,dx_j = (2t)^{-1}.
\end{equation}
Next, we consider $h = h_1 + h_2$ , where $h_k = (h_{k,1},\ldots , h_{k,n})^T, \ k=1,2,$
\begin{eqnarray}\nonumber
h_{1,j} &=&  \left( \sqrt{ - K(i \sqrt{t})} \theta_j(0,i\sqrt{t}) \right)^{-1} \left( \theta_j(x_j,i \sqrt{t}) - e^{-\sqrt{t} x_j} \right), \\  \nonumber
h_{2,j} &=& \left( \left( \sqrt{ - K(i\sqrt{t})} \theta_j(0,i \sqrt{t}) \right)^{-1} - \left(\sqrt{n} t^{1/4} \right)^{-1} \right) e^{-\sqrt{t} x_j}.
\end{eqnarray}
Because of \eqref{wK} we have  as $t \to \infty$,
\begin{equation}\label{h_1}
\left| \sqrt{ - K(i \sqrt{t})} \theta_j(0,i\sqrt{t}) \right|^{-1} = O(t^{-1/4})
\end{equation}
and further, see e.g. Lemma 4.1.4. \cite{Y},
\begin{equation}\label{h_2}
\left| \theta_j(x_j,i \sqrt{t}) - e^{-\sqrt{t} x_j} \right| \leq \frac{c}{\sqrt{t}} e^{-\sqrt{t} x_j}, \quad c \in \R.
\end{equation}
Hence, in view of \eqref{h_1} and \eqref{h_2}, $\| h_1 \| = O(t^{-1})$. To compute the asymptotics for $h_2$ we rewrite
\begin{eqnarray}\nonumber
h_{2,j} &=& \left( \sqrt{n} t^{1/4} - \sqrt{ - K(i\sqrt{t})} \theta_j(0,i \sqrt{t}) \right) \left( \sqrt{ - K(i\sqrt{t})} \theta_j(0,i \sqrt{t}) \right)^{-1} f_j(x_j)\\ \nonumber
&=&\frac{ \left( n \sqrt{t} + K(i \sqrt{t}) \theta_j^2(0, i \sqrt{t}) \right)}{ \sqrt{n} t^{1/4} + \sqrt{ - K(i\sqrt{t})} \theta_j(0,i\sqrt{t}) } \frac{f_j(x_j)}{\sqrt{ - K(i\sqrt{t})} \theta_j(0,i \sqrt{t})}.
\end{eqnarray}
Together with \eqref{f^2} and \eqref{h_1}, this leads to $\|h_2\| = O(t^{-1})$ and therefore, $\|h\| = O(t^{-1})$. This yields in view of \eqref{rank2norm} and \eqref{rank2} that $\| \mathcal{R} \|_1 = O(t^{-3/2})$. This proves the assertion of the lemma.
\end{proof}

With Lemma \ref{conditioncheck} all assumptions of Proposition \ref{SSF} are fullfilled and therefore the spectral shift function for the pair $H,H_0$ in $L_2(\Gamma)$ satisfies the relations \eqref{lnSSF} and \eqref{trSSF}. Especially, we have proved Theorem \ref{meins}.

\begin{remark}
Lemma \ref{conditioncheck} implies that $(H+c)^{\beta} - (H_0 +c )^{\beta}$ is trace class for $\beta < 1/2$, see \cite{Y}.
\end{remark}

\section{Low-energy asymptotics and Levinson's formula} \label{levii}

In this section we study the low-energy asymptotics of $D(z)$ as $|z| \to 0.$ This will allow us to prove an analog of Levinson's formula for star shaped quantum graphs. \textit{Throughout this section we assume that \eqref{firstmoment} is satisfied.}

\begin{definition}\label{zeroresonance}
We say that the operator $H$ in $L_2(\Gamma)$, given in \eqref{staroperator}, has a \textit{resonance} at $\zeta = 0$ if the equation
\begin{equation}\label{tavsan}
-u'' + Vu = 0
\end{equation}
has a non-trivial bounded solution satisfying the continuity and Kirchhoff conditions. By definition, the \textit{multiplicity} of the resonance is the dimension of the corresponding solution space.
\end{definition}

\begin{remark}
We shall show below that $\zeta= 0$ is never an eigenvalue.
\end{remark}

We recall some auxiliary results on half-line Schr\"odinger operators. For the half-line Schr\"odinger operator the Jost solution of the equation $H_{j} u = \zeta^2 u$ was characterized by its asymptotics $\theta_j(x,\zeta) = e^{ix \zeta} ( 1+o(1))$ as $| \zeta | \to \infty$. Further, the function $\theta_j(x,0)$ satisfies the equation $H_j u= 0$ and its behavior at $\zeta = 0$ is given by 
\begin{equation}\label{behaviorat0}
\theta_j(x,0) = 1 + O\left(\int_x^\infty y |V_j(y)| \,dy \right) = 1+o(1), \quad \mbox{as} \ x \to \infty,
\end{equation}
(see e.g. Lemma 4.3.1.,\cite{Y}). Recall also that the Jost function is $w_j(\zeta)= \theta_j(0,\zeta)$. If $w_j(0) = 0$, the low-energy asymptotics
\begin{equation}\label{loww}
w_j(\zeta)  = -i w_0^{(j)} \zeta + o(\zeta) ,\quad \zeta \to 0
\end{equation}
is true with some constant $w_0^{(j)} \neq 0$.

\begin{lemma}
If $\Theta(x, 0)$ is a bounded solution of equation \eqref{tavsan}, then for some $c_j \in \C$,
\begin{equation}\label{T}
\Theta(x, 0) = \bigoplus_{j=1}^n c_j \theta_j(x,0).
\end{equation}
Moreover, the operator $H$ cannot have a zero eigenvalue.
\end{lemma}

\begin{proof} 
If $\Theta(x,0)$ solves the equation \eqref{tavsan} with $\zeta = 0$, then the restriction of $\Theta(x,0)$ to the edge $e_j$ is a solution of the corresponding zero-energy equation on the half-line for every $1\leq j \leq n.$ As stated above, the function $\theta_j(x,0)$ solves the zero-energy equation and is bounded at infinity by \eqref{behaviorat0}. We note that $\theta_j(x,0)$ is the only solution of the zero-energy equation on the half-line, which is bounded at infinity. Indeed, the solution
$$
\tau_j(x,0) = \theta_j(x,0) \int_{x_0}^x \theta_j(y,0)^{-2} \,dy, \quad x \geq x_0
$$
which is linearly independent of $\theta_j(x,0)$ has as $x \to \infty,$ the asymptotics
$$
\tau_j(x,0) = x +o(x) 
$$
for all $1\leq j\leq n,$ (Lemma 4.3.2, \cite{Y}). Here $x_0$ is an arbitrary point such that $\theta(x) \neq 0$ for $x \geq x_0$.
Finally, equation \eqref{tavsan} cannot have a nontrivial solution belonging to $L_2(\Gamma)$ at infinity as $\Theta_j(x,0) = c_j +o(1)$ for $x \to \infty$.

\end{proof}

\begin{lemma}\label{M}
Let $M:= \# \{ j: \ w_j(0)= 0 \}$.
\begin{enumerate}
\item If $\zeta =0$ is not a resonance, then either $M=0$ and $K(0) \neq 0$ or $M=1$ and $K(\zeta)$ has a pole at $\zeta=0$.
\item Assume that $\zeta=0$ is a resonance of multiplicity $ m\geq 1$, then either\\
a) any resonance function vanishes at the vertex and $m = M-1 \geq 1$, or\\
b) the resonance is of multiplicity one, the corresponding resonance function is non-zero at the vertex and $m=1, \ M=0$ and $K(0) = 0$. 
\end{enumerate}
\end{lemma}

\begin{proof}
We first note that if $M \geq 1$, i.e. $\theta_j(0,0) = 0$ for some $j$, then any resonance function must vanish at $\zeta = 0$ because of the continuity condition.

\begin{enumerate}

\item
 If $M \geq 2$, then it is always possible to construct a zero-energy function by setting $c_j =0 $ if $w_j(0) \neq 0$ and determining the $c_j$'s such that the Kirchhoff vertex condition is fullfilled if $w_j(0) =0$. Hence, if $\zeta =0$ is not a resonance, then necessarily $M \leq 1.$ If $M=1$, then obviously $K(\zeta)$ has a pole at $\zeta = 0$. Moreover, if $M=0$ and $\zeta =0$ is not a resonance, then $K(0) \neq 0$, because if $K(\zeta)$ would vanish in $\zeta = 0$, then it would follow from
\begin{eqnarray}\label{K}
K(0) = \sum_{j=1}^n \frac{\theta'_j(0,0)}{\theta_j(0,0)} = \sum_{j=1}^n \frac{c_j \theta'_j(0,0)}{c_j \theta_j(0,0)} =0
\end{eqnarray}
that the function $\Theta(x,0)$ given in \eqref{T} is a zero-energy resonance function for suitable $c_j$. Hence, $K(0) \neq 0$ if $\zeta =0$ is not a resonance.

\item
a) If $\zeta  = 0$ is a resonance, then $M\neq1$, as it is not possible to construct a resonance function satisfying the vertex conditions and having support on only one edge of $\Gamma$. If $M \geq 2$, then because of the continuity condition any resonance function has to vanish at the vertex. Further, we set $c_j = 0$ for all $j$ with $w_j(0) \neq 0$, then there are $M-1$ linearly independent choices for the remaining $c_j$'s such that the Kirchhoff vertex condition is fullfilled. Hence, the multiplicity of the resonance function is $m=M-1 \geq 1.$\\
b) If $\zeta$ is a resonance with $M=0$, then $\theta_j(0,0) \neq 0$ for all $j$ and the coefficents $c_j$ are determined uniquely by the $n-1$ continuity conditions and the Kirchhoff vertex condition. Further, this implies because of \eqref{K} that $K(0) =0.$
\end{enumerate}
\end{proof}

In the following proposition, we give the low-energy asymptotics for $D(z)$ as $|z| \to 0.$

\begin{proposition}\label{proposition}
Let $m$ be the multiplicity of the resonance $\zeta = 0$, with the convention that $m=0$ if $\zeta = 0$ is not a resonance. If $m=1$, we assume in addition that condition \eqref{secondmoment} is satisfied for all $1 \leq j \leq n$.
Then, as $\zeta \to 0$,
\begin{equation}\label{lowenergyD}
D(z) = c \zeta^{m-1} (1+o(1)),\quad z = \zeta^2,
\end{equation}
with $c \neq 0$.
\end{proposition}

For the proof of Proposition \ref{proposition} we shall need the following

\begin{lemma}
Let $H_D = - d^2 / dx^2 +V(x)$ in $L_2([0,\infty))$ be given with Dirichlet boundary condition at the origin and assume that $\int_0^\infty (1+x) |V(x)| \,dx < \infty$. Let $\theta(x,\zeta)$ be the Jost solution on the half-line.
\begin{enumerate}
\item If $\theta(0,0) =0$, then 
\begin{equation}\label{trabzon}
\dot{\theta}(0,0) \theta'(0,0) = -i.
\end{equation}
\item If $\theta(0,0) \neq 0$ and $\int_0^\infty (1+x^2) |V(x)| \,dx < \infty$, then
\begin{equation}\label{trabzon2}
\dot{\theta}(0,0) \theta'(0,0) - \dot{\theta}'(0,0) \theta(0,0) = - i.
\end{equation}
\end{enumerate}  
\end{lemma}

\begin{proof}
If $\theta(0,0) = 0$, then $\dot{\theta}(0,0)$ is defined for $V$ having a first moment and $\dot{\theta}(0,0) = -i c_0,\  c_0 \neq 0$, by \cite{Y}, Proposition 4.3.7. Further, it was shown in \cite{Y} (4.3.11), that 
\begin{equation}\label{dizi1}
\varphi(x,0) = c_0 \theta(x,0),
\end{equation}
where $\varphi$ is the regular solution of $-u'' + Vu = \zeta^2 u$. Taking the derivative with respect to $x$ on both sides of \eqref{dizi1} and setting $x=0$ yields that $1= c_0 \theta'(0,0)$, as the regular solution of the Dirichlet problem was defined by the condition $\varphi'(0,\zeta) =1$. Hence, equation \eqref{trabzon} follows. If $\theta(0,0) \neq 0$, then $\dot{\theta}(0,0)$ is only defined for $V$ having a second moment and
 \begin{equation}\label{dizi2}
 \varphi(x,0) = i \dot{\theta}(0,0) \theta(x,0) - i \theta(0,0) \dot{\theta}(x,0)
 \end{equation}
 by \cite{Y}, Corollary 4.3.11. Again, taking on both sides of \eqref{dizi2} the derivative with respect to $x$ and setting $x=0$ leads to \eqref{trabzon2}. 

\end{proof}

\begin{proof}[Proof of Proposition \ref{proposition}]
We consider the explicit expression for the perturbation determinant $D(z)$ which was given in \eqref{Dstar},
\begin{equation}\label{Dstar2}
D(z) = \frac{  K(\zeta)}{ i n \zeta} \prod_{j=1}^n w_j(\zeta),\quad z= \zeta^2.
\end{equation}
Let us first consider the case $M=0$. Then by Lemma \ref{M} either $\zeta = 0$ is not a resonance and $K(0) \neq 0$ or $\zeta = 0$ is a resonance and $K(0) = 0$. If $\zeta =0$ is not a resonance, then we see from \eqref{Dstar2} that as $\zeta \to 0$, $D(z) = c \zeta^{-1} (1+o(1)), \ c \neq 0.$ If $\zeta =0$ is a resonance, then we consider $\dot{K}(0) = \sum_{j=1}^n \left( \dot{\theta}'_j(0,0) \theta_j(0,0) - \theta_j'(0,0) \dot{\theta}_j(0,0) \right) \theta_j^{-2}(0,0)$ which by \eqref{trabzon2} is the same as $\dot{K}(0) =\sum_{j=1}^n  i \theta_j^{-2}(0,0) \neq 0$. Hence, by applying l'Hospital we have as $\zeta \to 0$, $D(z) \to c \neq 0$. 

Next, we consider case $M \geq 1$. Without loss of generality let $\theta_1(0,0) = \ldots =\theta_M(0,0) = 0$. We rewrite \eqref{Dstar2} as 
\begin{equation}\label{D3}
D(z) = \frac{1}{in \zeta} \left( \sum_{j=1}^M \frac{\theta_j'(0,\zeta)}{\theta_j(0,\zeta)} \prod_{k=1}^n \theta_k(0,\zeta) + \sum_{j=M+1}^n \frac{\theta'_j(0,\zeta)}{\theta_j(0,\zeta)} \prod_{k=1}^n \theta_k(0,\zeta) \right), \ \  z = \zeta^2.
\end{equation}
Obviously, the second term on the right-hand side is $O(\zeta^M)$. In the first term on the right-hand side we have for each $1 \leq j \leq M$, as $\zeta \to 0$
\begin{eqnarray}\nonumber
 \frac{\theta_j'(0,\zeta)}{\theta_j(0,\zeta)} \prod_{k=1}^n \theta_k(0,\zeta) &=& \theta_j'(0,\zeta)\prod_{\substack{k = 1 \\ k \neq j}}^n \theta_k(0,\zeta) \\ \nonumber
 &=&  \zeta^{M-1}  \theta_j'(0,0) \prod_{\substack{k= 1 \\ k \neq j}}^M  \dot{\theta}_k(0,0) \prod_{k=M+1}^n \theta_k(0,0) + O(\zeta^M) \\ \nonumber
&=& \zeta^{M-1}  \frac{\theta_j'(0,0)}{\dot{\theta}_j(0,0)} \prod_{k = 1}^M \dot{\theta}_k(0,0) \prod_{k=M+1}^n \theta_k(0,0)  + O(\zeta^M).
\end{eqnarray}
In view of \eqref{trabzon} and \eqref{D3} we arrive at
\begin{eqnarray}\nonumber
D(z) &=&\frac{ \zeta^{M-1} }{in\zeta} \left( \prod_{k=1}^M \dot{\theta}_k(0,0)  \prod_{k=M+1}^n \theta_k(0,0) i \sum_{j=1}^M (\theta_j'(0,0))^2 + O(\zeta) \right)\\ \nonumber
& =& c \zeta^{M-2} + O(\zeta^{M-1}), \quad c \neq 0,\ z = \zeta^2.
\end{eqnarray}

\end{proof}
 
In the remaining part we prove an analog of Levinson's formula for star shaped graphs. 
For $k \in \R$, we set $D(k^2) =  a(k) e^{i\eta(k)}$, where $a(k) = |D(k^2)|$. Then it follows from the representation $D(k^2) = \left( K(k)/(ink) \right) \prod_{j=1}^n w_j(k)$ that 
\begin{equation}\label{symmetrieeta}
- \eta(k) = \eta(-k).
\end{equation}
Indeed, it follows from the uniqueness of the Jost solutions $\theta_j(x, \zeta)$ that $\overline{\theta_j(x,k)} = \theta_j(x,-k),$ and $\overline{\theta'_j(x,k)} = \theta'_j(x,-k)$, and hence also $\overline{w_j(k)} = \overline{\theta_j(0,k) } = w_j(-k)$. 

\begin{proof}[Proof of Theorem \ref{levinsonstar}]
First, we note that there is a spectral theoretical result relating the zeros of the perturbation determinant to the eigenfunctions of $H$ as follows. The function $D(\zeta^2)$ has a zero in $\zeta$ of order $r$ if and only if $\zeta^2$ is an eigenvalue of multiplicity $r$ of the operator $H$, \cite{Y2}. Obviously, the zeros of $D(\zeta^2)$ lie on the positive imaginary axis as $H$ is a self-adjoint operator and therefore it may have only real eigenvalues.

We apply the argument principle to the function $D(\zeta^2)$ and the contour $\Gamma_{R,\varepsilon}$ which consists of the half-circles $C_R^+ = \{ |\zeta| =R, \im \zeta \geq 0 \}$ and $C_{\varepsilon}^+ = \{ |\zeta| = \varepsilon , \ \im \zeta \geq 0 \}$ and the intervals $(\varepsilon, R)$ and $(-R,-\varepsilon)$. We choose $R$ and $\varepsilon$ such that all of the $N$ negative eigenvalues of $H$ lie inside the contour $\Gamma_{R,\varepsilon}$. The function $D(\zeta^2)$ is analytic inside and on $\Gamma_{R,\varepsilon}$ as $w_j(\zeta)$ is analytic in the the upper half-plane $\im \zeta >0.$ Thus,
\begin{equation}\label{argumentprinciplestar}
\int_{\Gamma_{R,\varepsilon}} \frac{\frac{d}{d\zeta} D(\zeta^2)}{ D(\zeta^2)} \,d\zeta = 2 \pi i N.
\end{equation}
Remember that $\lim_{  \im \zeta \to \infty } D( \zeta^2 ) = \lim_{ \im \zeta  \to \infty } K(\zeta) / (ni\zeta) = 1 + O(|\zeta|^{-1}).$
Thus, we can fix the branch of the function $\ln D(\zeta^2)$ by the condition $\ln D(\zeta^2) \to 0$ as $\im \zeta \to \infty.$ Then, we have $ \ln D(\zeta^2) = \ln |D(\zeta^2)| +i \mbox{arg}\ D(\zeta^2)$. Equation \eqref{argumentprinciplestar} implies that
\begin{equation}\label{combine3}
\mbox{var}_{\Gamma_{R,\varepsilon}} \mbox{arg}\ D(\zeta^2) = 2 \pi N.
\end{equation}
We define $\eta(0) := \lim_{k \to 0+} \eta(k)$. This limit exists because of asymptotics \eqref{lowenergyD}. It follows with \eqref{symmetrieeta} that,
$$
\mbox{var}_{\Gamma_{R,\varepsilon}} \mbox{arg}\ D(\zeta^2)=2 (\eta(R) - \eta(\varepsilon)) + \mbox{var}_{C^+_R} \mbox{arg}\ D(\zeta^2)+  \mbox{var}_{C^+_{\varepsilon}} \mbox{arg}\ D(\zeta^2).
$$
\\

\vspace*{2.6cm}
\begin{center}
\begin{figure}[h!]
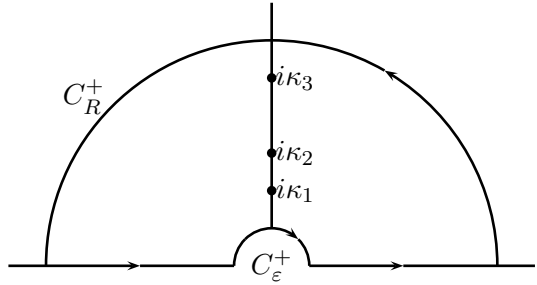

\psset{unit=0.5cm}
\psarc[linewidth=1pt]{-}(0,0){1}{0}{45}
\psarc[linewidth=1pt]{<-}(0,0){1}{45}{135}
\psarc[linewidth=1pt](0,0){1}{135}{180}
\psarc[linewidth=1pt]{->}(0,0){6}{0}{60}
\psarc[linewidth=1pt](0,0){6}{60}{180}
\psline[linewidth=1pt]{->}(1,0)(3.5,0)
\psline[linewidth=1pt]{-}(3.5,0)(7,0)
\psline[linewidth=1pt]{->}(-7,0)(-3.5,0)
\psline[linewidth=1pt]{-}(-3.5,0)(-1,0)
\psline[linewidth=1pt]{-}(0,1)(0,7)
\psdot(0,3)
\psdot(0,2)
\psdot(0,5)
\rput(0.5,2){$\ i\kappa_1$}
\rput(0.5,3){$\ i\kappa_2$}
\rput(0.5,5){$\ i\kappa_3$}
\rput(-5,4.5){$C_R^+$}
\rput(0,0){$C_{\varepsilon}^+$}
\caption{contour of integration $\Gamma_{R,\varepsilon}$}
\end{figure}
\end{center}

\noindent Now, we let $R \to \infty$ and $\varepsilon \to 0.$ Because of \eqref{limitDstar}, $\lim_{R \to \infty} \mbox{var}_{C^+_R} \mbox{arg}\ D(\zeta^2) = 0$. Hence, it follows from \eqref{combine3} that
$$
\eta(\infty) - \eta( \varepsilon) = \pi N - \frac{1}{2} \mbox{var}_{C^+_{\varepsilon}}  \mbox{arg} \ D(\zeta^2).
$$
By Proposition \ref{proposition}, $ \lim_{\varepsilon \to 0} \mbox{var}_{C^+_{\varepsilon}} \mbox{arg} \ D(\zeta^2) = -(m-1) \pi.$
Thus, 
\begin{equation}\label{levistar}
\eta(\infty) - \eta(0) = \pi \left( N + \frac{m-1}{2} \right).
\end{equation}
Note that $\eta(\infty) = 0$ since $\ln D(\zeta^2) \to 0$ as $|\zeta| \to \infty$.
It remains to note that in view of Theorem \ref{meins}, the following identity is true for $\lambda = k^2,\ k>0$,
\begin{equation}\label{levistar2}
\xi(\lambda) = \pi^{-1} \lim_{\varepsilon \to 0+} \mbox{arg} \ D(\lambda +i \varepsilon) =  \pi^{-1} \eta(\lambda^{1/2}),\quad \lambda > 0.
\end{equation}
The assertion of Theorem \ref{levinsonstar} then follows by combining \eqref{levistar} and \eqref{levistar2}. 

\end{proof}

{\bf Acknowledgments}. The author is grateful to Timo Weidl for valuable comments. Many thanks to Rupert Frank for fruitful discussions and references.

\bibliographystyle{plain}
\bibliography{d.bib}
\end{document}